\documentclass[12pt]{amsart}
\usepackage{mathrsfs}

\font\bbbld=msbm10 scaled\magstephalf
\newcommand{\bi}{\bar{i}}
\newcommand{\bj}{\bar{j}}
\newcommand{\bk}{\bar{k}}
\newcommand{\bl}{\bar{l}}

\newcommand{\bn}{\bar{n}}
\newcommand{\bp}{\bar{p}}
\newcommand{\bq}{\bar{q}}

\newcommand{\bz}{\bar{z}}

\newcommand{\bM}{\bar{M}}

\newcommand{\bpartial}{\bar{\partial}}

\def \a{\alpha}

\def \p{\partial}

\newcommand{\fg}{\mathfrak{g}}

\newcommand{\fRe}{\mathfrak{Re}}

\newcommand{\bfC}{\hbox{\bbbld C}}

\newcommand{\bfR}{\hbox{\bbbld R}}
\newcommand{\bfS}{\hbox{\bbbld S}}

\newcommand{\tr}{\mbox{tr}}

\newcommand{\ol}{\overline}
\newcommand{\ul}{\underline}

\newtheorem{theorem}{Theorem}[section]
\newtheorem{lemma}[theorem]{Lemma}
\newtheorem{proposition}[theorem]{Proposition}

 \theoremstyle{definition}

\theoremstyle{remark}

\numberwithin{equation}{section}

\begin{document}
\setlength{\baselineskip}{1.2\baselineskip}

\title[Fully Nonlinear Elliptic Equations]
{On a Class of Fully Nonlinear Elliptic Equations \\ on Hermitian Manifolds}

\author{Bo Guan and Wei Sun}
\address{Department of Mathematics, Ohio State University,
         Columbus, OH 43210}
\email{guan@math.ohio-state.edu \\ sun@math.ohio-state.edu}

\begin{abstract}
We derive {\em a priori} $C^2$ estimates for a class of complex
Monge-Amp\`ere type equations on
Hermitian manifolds. As an application we solve the Dirichlet problem
for these equations under the assumption of existence of a subsolution;
the existence result, as well as the second order boundary estimates,
is new even for bounded domains in $\bfC^n$.

\end{abstract}

\maketitle

\section{Introduction}

Let $(M^n,\omega)$ be a compact Hermitian manifold of dimension $n \geq 2$ with
smooth boundary $\partial M$ and $\chi$ a smooth real $(1,1)$ form on
$\bar M := M \cup \partial M$.
Define for a function $u \in C^2 (M)$,
\[ \chi_u = \chi + \frac{\sqrt{-1}}{2} \partial \bpartial u \]
and set
\[ [\chi] = \big\{\chi_u:  \, u \in C^2 (M)\big\},
  \;\; [\chi]^+ = \big\{\chi' \in [\chi]: \chi' > 0\}. \]

In this paper we are concerned with %the Dirichlet problem of
the equation for $1 \leq \alpha \leq n$,
\begin{equation}
\label{CH-I10}
%\left\{
\begin{aligned}
    \chi_u^n = & \psi \chi_u^{n-\alpha} \wedge \omega^{\alpha}
      \;\; \mbox{in $M$, }.
%          u = & \varphi \;\; \mbox{on $\p M$}.
\end{aligned}
%\right.
\end{equation}

We require $\chi_{u} > 0$ so that equation~\eqref{CH-I10} is elliptic;
we call such functions {\em admissible} or $\chi$-{\em plurisubharmonic}.
Consequently, we assume $\psi > 0$ on $\bM$; equation~\eqref{CH-I10} becomes
degenerate when $\psi \geq 0$.

When $\alpha = n$ this is the complex Monge-Amp\`ere equation which plays
extremely important roles in complex geometry and analysis, especially in
K\"ahler geometry, and has received extensive study since the
fundamental work of Yau~\cite{Yau78} (see also \cite{Aubin78}) on compact
K\"ahler manifolds and that of Caffarelli, Kohn, Nirenberg and Spruck~\cite{CKNS}
for the Dirichlet problem in strongly pseudoconvex domains in $\bfC^n$.
For $\alpha =1$ equation~\eqref{CH-I10} also arises naturally in geometric
problems; it was posed by Donaldson~\cite{Donaldson99a} in
connection with moment maps and is closely related to the Mabuchi
energy \cite{Chen04}, \cite{Weinkove06}, \cite{SW08}.

Donaldson's problem assumes $M$ is closed, both $\omega$, $\chi$ are K\"ahler
and $\psi$ is constant. It was studied by Chen~\cite{Chen04},
Weinkove~\cite{Weinkove04}, \cite{Weinkove06}, Song and Weinkove~\cite{SW08}
using parabolic methods.
In \cite{SW08} Song and Weinkove give a necessary and sufficient solvability
condition. Their result was extended by Fang, Lai and Ma~\cite{FLM11} to all
$1 \leq \alpha < n$.

In this paper we study the Dirichlet problem for equation~\eqref{CH-I10}
on Hermitian manifolds.
Given $\psi \in C^{\infty} (\bM)$ and $\varphi \in C^{\infty} (\partial M)$,
we wish to find a solution $u \in C^{\infty} ({\bar M})$ of equation~\eqref{CH-I10}
satisfying the boundary condition
 \begin{equation}
\label{CH-I10b}
 u = \varphi \;\; \mbox{on $\partial M$}.
\end{equation}

The Dirichlet problem for the complex Monge-Amp\`ere equation in $\bfC^n$
was studied by Caffarelli, Kohn, Nirenberg and Spruck~\cite{CKNS}
%\[ \det u_{i\bj} = \psi \;\; \mbox{in $\Omega$},  \,\,
%    u = \varphi \;\; \mbox{on $\partial \Omega$} \]
on strongly pseudoconvex domains. % $\Omega$ in $\bfC^n$.
Their result was extended to Hermitian manifolds by Cherrier and Hanani 
%\cite{Cherrier87},
\cite{CH99}, \cite{Hanani96a}, and by the first author~\cite{Guan98b}
to arbitrary bounded domains in $\bfC^n$ under the
assumption of existence of a subsolution. See also the more recent papers \cite{GL10},
\cite{Zhang10}, and related work of Tosatti and Weinkove~\cite{TWv10a},
\cite{TWv10b} who completely extended the zero order estimate of Yau~\cite{Yau78}
on closed K\"ahler manifolds to the Herimatian case.
In \cite{LiSY04} Li treated the Dirichlet problem for more general fully
 nonlinear elliptic equations in $\bfC^n$ but needed to assume the existence of
a {\em strict} subsolution. Li's result does not cover
equation~\eqref{CH-I10} as it fails to satisfy some of the key structure conditions
in \cite{LiSY04}.

In this paper we prove the following existence result which is new even in the case
when $M$ is a bounded domain in $\bfC^n$ and $\chi = 0$; we assume
$2 \leq \alpha \leq n-2$ as the cases $\alpha = 1$ and $\alpha = n-1$ were considered
in \cite{GL12} and \cite{GL3}, while for the complex Monge-Amp\`ere equation
($\alpha = n$) it was proved in \cite{GL10}.

\begin{theorem}
\label{gl-thm20}
Let $\psi \in C^{\infty} (\bM)$, $\psi > 0$ and $\varphi \in C^{\infty} (\partial M)$.
There exists a unique admissible solution $u \in C^{\infty} (\bM)$ of the Dirichlet
problem~\eqref{CH-I10}-\eqref{CH-I10b},
provided that there exists an admissible subsolution $\ul{u} \in C^2 (\bM)$:
 \begin{equation}
\label{CH-I20b}
 \left\{ \begin{aligned}
& \chi_{\ul{u}}^n \geq  \psi \chi_{\ul{u}}^{n-\alpha} \wedge \omega^{\alpha}
\;\; \mbox{on $\bM$}  \\
& \ul{u} = \varphi \;\; \mbox{on $\partial M$}.
  \end{aligned}  \right.
\end{equation}
\end{theorem}

In order to solve the Dirichlet problem~\eqref{CH-I10}-\eqref{CH-I10b}
one needs to derive {\em a priori} $C^2$ estimates up to the boundary for admissible
solutions. The most difficult step is probably the second order estimates on the
boundary.

\begin{theorem}
\label{gl-thm20bc2}
Suppose $\psi \in C^1 (\bM)$, $\psi > 0$ and $\varphi \in C^4 (\partial M)$
and $\ul{u} \in C^2 (\bM)$ is an admissible subsolution satisfying
\eqref{CH-I20b}.
Let $u \in C^3 (\bM)$ be an admissible solution of the Dirichlet
problem~\eqref{CH-I10}-\eqref{CH-I10b}. Then
\begin{equation}
 \label{cma-37}
\max_{\partial M} |\nabla^2 u| \leq C
\end{equation}
where $C$ depends on $|u|_{C^1 (\bM)}$, $\min \psi^{-1}$, $|\ul{u}|_{C^2 (\bM)}$
and $\min \{c_1: c_1 \chi_{\ul u} \geq \omega\}$, as well as other known data.
\end{theorem}

This estimate is new for domains in $\bfC^n$. Note that $\partial M$
is assumed to be smooth and compact in Theorem~\ref{gl-thm20bc2}, but otherwise
is completely arbitrary. In general, the Dirichlet
problem~\eqref{CH-I10}-\eqref{CH-I10b} is not always solvable in an
arbitrary smooth bounded domain in $\bfC^n$ without the subsolution assumption.
In the theory of nonlinear elliptic equations, many well
known classical results assume certain geometric conditions on the boundary of the
underlying domain; see e.g. \cite{Serrin69}, \cite{CKNS}, \cite{CNS1} and \cite{CNS3}.
In \cite{GS93}, \cite{Guan98a} and \cite{Guan98b}, J. Spruck and the first author
were able to solve the Dirichlet problem for real and complex Monge-Amp\`ere
equations on arbitrary smooth bounded domains assuming the existence of a
subsolution. Their work was motivated by applications to geometric problems and
had been found useful in some important problems such as the proof by
P.-F. Guan~\cite{GuanPF02}, ~\cite{GuanPF08} of the Chern-Levine-Nirenberg
conjecture~\cite{CLN69}, and work on the Donaldson conjectures~\cite{Donaldson99}
on geodesics in the space of K\"ahler metrics; we refer the reader to \cite{PSS12}
for recent progress and further references on this fast-developing subject.

On a closed K\"ahler manifold $(M, \omega)$,
Fang, Lai and Ma~\cite{FLM11} proved second and zero order estimates
for equation~\eqref{CH-I10} when $\chi$ is also K\"ahler and $\psi$ is constant.
We extend their second order estimates to Hermitian manifolds and for general $\chi$
and $\psi$. Technically the major difficulty is to control extra third
order terms which occur due to the nontrivial torsion of the Hermitian metric.
This was done in \cite{GL12}, \cite{GL3} for $\alpha = 1$ and $\alpha = n-1$;
the case $2 \leq \alpha \leq n-2$ is considerably more complicated. In order
to solve the Dirichlet problem we also need global gradient estimates.
Following \cite{SW08} and \cite{FLM11} let
\begin{equation}
\label{CH-I20}
 \mathscr{C}_{\alpha} (\omega)
   = \big\{[\chi]: \, \exists \, \chi' \in [\chi]^+, \,
        n\chi'^{n-1} > (n-\alpha) \psi \chi'^{n-\alpha -1}\wedge \omega^\alpha\big\}.
\end{equation}

\begin{theorem}
\label{cmate-main-1}
Let $u\in C^4(M) \cap C^2 (\bar M) $ be an admissible solution of
equation~(\ref{CH-I10}) where $\psi \in C^2(\bar M)$, $\psi >0$.
Suppose that $\chi \in \mathscr{C}_{\alpha} (\omega)$.
Then there are constants $C_1, C_2$ depending on $|u|_{C^0(\bar M)}$ such that
\begin{equation}
    \label{CH-I30g}
    \sup_M |\nabla u| \leq C_1 (1 + \sup_{\p M}|\nabla u|),
\end{equation}
\begin{equation}
    \label{CH-I30s}
    \sup_M \Delta u  \leq C_2(1 + \sup_{\p M} \Delta u).
\end{equation}
In particular, if $M$ is closed ($\partial M = \emptyset$) then
$ |\nabla u| \leq C_1$ and $|\Delta u | \leq C_2$ on $M$.
\end{theorem}

The cone $\mathscr{C}_{\alpha} (\omega)$ was first introduced by
Song and Weinkove~\cite{SW08} ($\alpha=1$) and
Fang, Lai and Ma~\cite{FLM11} who derived the
estimate \eqref{CH-I30s}
%was derived by Song and Weinkove~\cite{SW08}
%for $\alpha=1$ and by Fang, Lai and Ma~\cite{FLM11}
on a closed K\"ahler manifold
$(M, \omega)$ when $\chi$ is also K\"ahler and
 \[ \psi = c_{\alpha} :=
 \frac{\int_M \chi^n}{\int_M \chi^{n-\alpha} \wedge \omega^{\alpha}},\]
which is a K\"ahler class invariant.
As in ~\cite{SW08}, ~\cite{FLM11} the constant $C_2$ in
Theorem~\ref{cmate-main-1} is independent of gradient
bounds, i.e. $C_2$ is independent of $C_1$.

The subsolution assumption \eqref{CH-I20b} implies
$[\chi] \in \mathscr{C}_{\alpha} (\omega)$. On a closed manifold, a subsolution
must be a solution or the equation has no solution. This is a consequence of
the maximum principle and a concavity property of equation~(\ref{CH-I10}).

The gradient estimate \eqref{CH-I30g} is crucial to the proof of
Theorem~\ref{gl-thm20} and is also new when $\omega$ and $\chi$ are
K\"ahler. Indeed, deriving gradient estimates for fully nonlinear equations
on complex manifolds turns out to be a rather challenging and mostly open
question. Only very recently were Dinew and Kolodziej~\cite{DK}
able to prove the gradient estimate using scaling techniques and Liouville
type theorems for the complex Hessian equation
\begin{equation}
\label{CH-I10H}
%\left\{
\begin{aligned}
    \omega^n = & \psi \chi_u^{n-\alpha} \wedge \omega^{\alpha}
%          u = & \varphi \;\; \mbox{on $\p M$}.
\end{aligned}
%\right.
\end{equation}
 on closed K\"ahler manifolds which is consequently solvable due to the earlier work of
 Hou, Ma and Wu~\cite{HMW10}.

The proof of Theorem~\ref{cmate-main-1} is carried out in
Sections~\ref{cmate-C1} and \ref{cmate-C2} where we derive the estimates
for $|\nabla u|$ and $\Delta u$, the gradient and Laplacian of $u$, respectively.
In Section~\ref{gblq-B} we establish the boundary estimates for second derivatives.
These estimates allow us to derive global estimates for all (real) second
derivatives as in Section 5 in \cite{GL10} and apply the Evans-Krylov theorem
since equation ~\eqref{CH-I10} becomes uniformly elliptic. Theorem~\ref{gl-thm20}
may then be proved by the continuity method. These steps are all well understood
so we shall omit them. In section~\ref{cmate-P} we recall some formulas on
Hermitian manifolds.

\bigskip

\section{Preliminaries}
\label{cmate-P}
\setcounter{equation}{0}
\medskip

Let $g$ and $\nabla$ denote the Riemannian metric and Chern connection of $(M, \omega)$.
The torsion and curvature tensors of $\nabla$ are defined by
\begin{equation}
\label{cma-K95}
\begin{aligned}
   T (u, v) \,& = \nabla_u v - \nabla_v u - [u,v], \\
 R (u, v) w \,& = \nabla_u \nabla_v w - \nabla_v \nabla_u w - \nabla_{[u,v]} w,
\end{aligned}
\end{equation}
respectively.
Following the notations in \cite{GL10}, in local coordinates $z = (z_1, \ldots, z_n)$
we have
\begin{equation}
\label{cma-K70}
\left\{ \begin{aligned}
 g_{i \bj} \,& = g \Big(\frac{\partial}{\partial z_i},
 \frac{\partial}{\partial \bz_j}\Big), \;\; \{g^{i\bj}\} = \{g_{i\bj}\}^{-1}, \\
  T^k_{ij} \,& = \Gamma^k_{ij} - \Gamma^k_{ji}
   = g^{k\bl} \Big(\frac{\partial g_{j\bl}}{\partial z_i}
       - \frac{\partial g_{i\bl}}{\partial z_j}\Big),  \\
 R_{i\bj k\bl} \,&   =  - g_{m \bl} \frac{\partial \Gamma_{ik}^m}{\partial \bz_j}
         = - \frac{\partial^2 g_{k\bl}}{\partial z_i \partial \bz_j}
       + g^{p\bq} \frac{\partial g_{k\bq}}{\partial z_i}
                  \frac{\partial g_{p\bl}}{\partial \bz_j}.
\end{aligned} \right.
\end{equation}

Recall that for a smooth function $v$, $v_{i\bj} = v_{\bj i} = \partial_i \bpartial_j v$,
$v_{i\bj k} = \partial_k v_{i\bj} - \Gamma_{ki}^l v_{l\bj}$ and
\[ v_{i\bj k\bl}
   = \bpartial_l v_{i\bj k} - \ol{\Gamma_{lj}^q} v_{i\bq k}. \]
We have (see e.g. \cite{GL12}),
\begin{equation}
\label{gblq-B145}
\left\{
\begin{aligned}
 v_{i \bj k} - v_{k \bj i} = \,& T_{ik}^l v_{l\bj},  \\
v_{i \bj \bk} - v_{i \bk \bj} = \,& \ol{T_{jk}^l} v_{i\bl},
\end{aligned}
 \right.
\end{equation}
\begin{equation}
\label{gblq-B147}
\left\{ \begin{aligned}
v_{i\bj k\bl} - v_{i\bj \bl k}
      = \,& g^{p\bq} R_{k\bl i\bq} v_{p\bj}
          - g^{p\bq} R_{k \bl p \bj} v_{i\bq}, \\
v_{i \bj k \bl} - v_{k \bl i \bj}
  = \,&  g^{p\bq} (R_{k\bl i\bq} v_{p\bj} - R_{i\bj k\bq} v_{p\bl})
        + T_{ik}^p v_{p\bj \bl} + \ol{T_{jl}^q} v_{i\bq k}
        - T_{ik}^p \ol{T_{jl}^q} v_{p\bq}.
 \end{aligned}  \right.
\end{equation}

Let $u \in C^4 (M)$ be an admissible solution of equation~\eqref{CH-I10}.
As in \cite{GL10} and \cite{GL12},
we denote $\fg_{i\bj} = \chi_{i\bj} + u_{i\bj}$,
$\{\fg^{i\bj}\} =  \{\fg_{i\bj}\}^{-1}$
and $w = \tr \chi + \Delta u$.
Note that $\{\fg_{i\bj}\}$ is positive definite.
Assume at a fixed point $p \in M$ that $g_{i\bj} = \delta_{ij}$
and $\fg_{i\bj}$ is diagonal. Then
\begin{equation}
\label{gblq-R155a}
 \begin{aligned}
u_{i \bi k \bk} - u_{k \bk i \bi}
   = \,& R_{k\bk i\bp} u_{p\bi} - R_{i\bi k\bp} u_{p\bk}
  + 2 \fRe\{\ol{T_{ik}^j} u_{i\bj k}\} -  T_{ik}^p \ol{T_{ik}^q} u_{p\bq},
  \end{aligned}
\end{equation}
and therefore,
\begin{equation}
\label{gblq-R155}
 \begin{aligned}
\fg_{i \bi k \bk} - \fg_{k \bk i \bi}
   = \,& R_{k\bk i\bi} \fg_{i\bi} - R_{i\bi k\bk} \fg_{k\bk}
         + 2 \fRe\{\ol{T_{ik}^j} \fg_{i\bj k}\}
         - |T_{ik}^j|^2 \fg_{j\bj} - G_{i\bi k\bk}
  \end{aligned}
 \end{equation}
where
 \begin{equation}
\begin{aligned}
  G_{i\bi k\bk}
   = \,& \chi_{k \bk i \bi} - \chi_{i \bi k \bk}
     +  R_{k\bk i\bp} \chi_{p\bi} - R_{i\bi k\bp} \chi_{p\bk}
     + 2 \fRe\{\ol{T_{ik}^j} \chi_{i\bj k}\}
     - T_{ik}^p \ol{T_{ik}^q} \chi_{p\bq}.
       \end{aligned}
\end{equation}

Let $S_k (\lambda)$ denote the $k$-th elementary symmetric polynomial
of $\lambda \in \bfR^n$
\[ S_k (\lambda) = \sum_{1 \leq i_1 < \cdots < i_k \leq n}
     \lambda_{i_1} \cdots \lambda_{i_k}. \]
In local coordinates we can write equation~\eqref{CH-I10} in the form
\begin{equation}
\label{cmate-M10'}
 F (\fg_{i\bj}) := \Big(\frac{S_n (\lambda_* (\fg_{i\bj}))}
                    {S_{n-\alpha} (\lambda_* (\fg_{i\bj} ))}\Big)^{\frac{1}{\alpha}}
                 = \Big(\frac{\psi}{C^\a_n}\Big)^{\frac{1}{\alpha}}
\end{equation}
or equivalently,
\begin{equation}
\label{cmate-M10}
    C^\a_n \psi^{-1} = S_\alpha (\lambda^* (\fg^{i\bj}))
\end{equation}
where $\lambda_* (A)$ and $\lambda^* (A)$ denote the eigenvalues of a Hermitian matrix
$A$ with respect to $\{g_{i\bj}\}$ and to $\{g^{i\bj}\}$, respectively.
Unless otherwise indicated we shall use $S_\alpha$ to denote
$S_\alpha(\lambda^* (\fg^{i\bar j}))$ when no possible confusion would occur.
We shall also occasionally write $F (\chi_u) := F (\fg_{i\bj})$ and
$F (\chi_{\ul u}) := F (\ul u_{i\bj} + \chi_{i\bj})$, etc.

Differentiating equation \eqref{cmate-M10} twice at a point $p$ where
$g_{i\bj} = \delta_{ij}$ and $\fg_{i\bj}$ is diagonal, we obtain
\begin{equation}
    \label{cmate-eq148}
    C^\a_n \partial_l (\psi^{-1}) = - \sum_i S_{\a -1;i} (\fg^{i\bi})^2 \fg_{i\bi l}
\end{equation}
and
\begin{equation}
    \label{cmate-eq149'}
\begin{aligned}
    C^\a_n \bpartial_l\partial_l (\psi^{-1})
   = \,& - \sum_i S_{\alpha -1;i}(\fg^{i\bar i})^2 \fg_{i\bar il\bar l}
        + \sum_{i,j} S_{\alpha -1;i}  (\fg^{i\bar i})^2 \fg^{j\bar j}
         (\fg_{i\bar j l} \fg_{j\bar i\bar l} + \fg_{j\bar i l} \fg_{i\bar j\bar l}) \\
       & + \sum_{i \neq j} S_{\alpha -2;ij} (\fg^{i\bar i})^2(\fg^{j\bar j})^2
   \big(\fg_{i \bar i l}\fg_{j \bar j\bar l} - \fg_{j\bar i l} \fg_{i\bar j\bar l}\big)
\end{aligned}
\end{equation}
where for $\{i_1,\cdots,i_s\} \subseteq \{1,\cdots,n\}$,
\[ S_{k;i_1\cdots i_s} (\lambda) = S_k (\lambda|_{\lambda_{i_1} = \cdots = \lambda_{i_s}= 0}). \]

We need the following inequality from \cite{GLZ}; see also Proposition 2.2
in \cite{FLM11},
%\begin{proposition}
%    \label{cmate-prop1}
%    For $\lambda\in \Gamma_n$, $\xi = (\xi_1,\cdots\xi_n) \in \bfC^n$,
 \begin{equation}
 \label{glz}
     \sum^n_{i=1} \frac{S_{\a -1;i}(\lambda)}{\lambda_i} \xi_i \bar\xi_i
       + \sum_{i,j} S_{\a -2;ij}(\lambda) \xi_i\bar\xi_j
  \geq \sum_{i,j} \frac{S_{\a -1;i}(\lambda)S_{\a -1;j}(\lambda)}
          {S_\a (\lambda)}\xi_i\bar\xi_j \geq 0
    \end{equation}
%\end{proposition}
for $\lambda = (\lambda_1, \ldots, \lambda_n)$, $\lambda_i > 0$ and
$(\xi_1, \ldots, \xi_n) \in \bfC^n$.
Apply \eqref{glz} to $\lambda_i = \fg^{i\bi}$, $\xi_i = (\fg^{i\bar i})^2 \fg_{i \bar i l}$
and sum over $l$.
We see that
\begin{equation}
\label{cmate-C2-E2}
	\sum_{i,l} S_{\alpha-1;i}(\fg^{i\bar i})^3\fg_{i\bar il}\fg_{i\bar i\bar l}
          + \sum_{i\neq j} \sum_{l} S_{\alpha-2;ij}
      (\fg^{i\bar i})^2(\fg^{j\bar j})^2 \fg_{i\bar il} \fg_{j\bar j\bar l}  \geq 0.
\end{equation}
Note also that
\[ \sum_{i \neq j} (S_{\alpha -1;i} - S_{\alpha-2;ij} \fg^{j\bar j})
       (\fg^{i\bar i})^2 \fg^{j\bar j} \fg_{j\bar i l} \fg_{i\bar j\bar l} \geq 0. \]
We obtain from \eqref{cmate-eq149'},
\begin{equation}
    \label{cmate-eq149}
\begin{aligned}
 \sum_i S_{\alpha -1;i}(\fg^{i\bar i})^2 \fg_{i\bar il\bar l}
   \geq \,& \sum_{i,j} S_{\alpha -1;i}  (\fg^{i\bar i})^2 \fg^{j\bar j}
         \fg_{i\bar j l} \fg_{j\bar i\bar l} - C.
\end{aligned}
\end{equation}

 Let $\ul u \in C^2 (\bM)$, $\chi_{\ul u}> 0$  such that
 \begin{equation}
    \label{cmate-eq14}
 n \chi_{\ul u}^{n-1} > (n-\alpha) \psi \chi_{\ul u}^{n-\alpha -1} \wedge \omega^\alpha.
 \end{equation}

 Thus there is $\epsilon > 0$ such that
\begin{equation}
\label{cmate-eq151}
    \epsilon\omega \leq \chi_{\ul u} \leq \epsilon^{-1} \omega.
\end{equation}

The key ingredient of our estimates in the following sections is
the following lemma.

\begin{lemma}
\label{cmate-2alternative}
There exist constants $N, \theta > 0$ such that when $w \geq N$
at a point $p$ where
$g_{i\bj} = \delta_{ij}$ and $\fg_{i\bj}$ is diagonal,
\begin{equation}
\label{GSn-P50}
\sum_i S_{\alpha -1;i} (\fg^{i\bar i})^2 (\ul u_{i\bar i} - u_{i\bar i})
\geq \theta \sum_i S_{\alpha -1;i} (\fg^{i\bar i})^2 + \theta
\end{equation}
and, equivalently,
\begin{equation}
\label{GSn-P50'}
\sum_{i, j} F^{i\bj} (\ul u_{i\bar j} - u_{i\bar j})
\geq \theta \sum_{i, j} F^{i\bj} g_{i\bar j} + \theta.
%  \sum_i S_{\alpha -1;i} (\fg^{i\bar i})^2 (\chi_{i\bar i} + \ul u_{i\bar i})
%        \geq (1+\theta) \alpha S_\alpha.
\end{equation}
\end{lemma}

Here and in the rest of this paper,
\[ F^{i\bar j} = \frac{\p F}{\p \fg_{i\bar j}} (\fg_{i\bj}). \]
It is well known that $\{F^{i\bar j}\}$ is positive definite.

An equivalent form of Lemma~\ref{cmate-2alternative} and its proof are given
in \cite{FLM11} (Theorem 2.8); see also \cite{Guan2012b} where it is proved for
more general fully nonlinear equations.
%One can also find a proof for more general fully nonlinear equations in \cite{Guan2012a}
%and \cite{Guan2012b}.
So we shall omit the proof here.

\bigskip

\section{The gradient estimates}
\label{cmate-C1}
\setcounter{equation}{0}

\medskip

In this section we establish the {\em a priori} gradient estimates.

\begin{proposition}
Suppose $\chi \in \mathscr{C}_{\alpha} (\omega)$ and
let $u \in C^3 (M) \cap C^1 (\bM)$ be an admissible solution of
\eqref{CH-I10}. 
There is a uniform constant $C > 0$ such that
\begin{equation}
\label{cmate-C1-1}
\sup_{\bM} |\nabla u | \leq C(1 + \sup_{\partial M}|\nabla u|).
\end{equation}
\end{proposition}
\begin{proof}
Let $\ul u \in C^2 (\bM)$, $\chi_{\ul u}> 0$ satisfy \eqref{cmate-eq14}
and consider $\phi = Ae^\eta$ where 
\[ \eta = \ul u - u + \sup_M (u - \ul u) \]
and $A$ is a constant to be determined.
Suppose the function $e^{\phi}|\nabla u|^2$ attains its maximal value at an
 interior point $p\in M$. 
Choose local coordinate around $p$ such that
$g_{i\bar j} = \delta_{ij}$ and $\fg_{i\bar j}$ is diagonal at $p$.
At $p$ we have
\begin{equation}
\label{cmate-C0-1}
    \frac{\partial_i(|\nabla u|^2)}{|\nabla u|^2} + \partial_i \phi = 0, \;\;
 \frac{\bar\partial_i(|\nabla u|^2)}{|\nabla u|^2} + \bar\partial_i\phi = 0
\end{equation}
and
\begin{equation}
\label{cmate-C0-2}
    \frac{\bar\partial_i\partial_i(|\nabla u|^2)}{|\nabla u|^2} 
     - \frac{\partial_i(|\nabla u|^2)\bar\partial_i(|\nabla u|^2)}{|\nabla u|^4} 
     + \bar\partial_i\partial_i\phi \leq 0.
\end{equation}

By direct computation,
\begin{equation}
\label{cmate-C0-3}
    \partial_i (|\nabla u|^2) = \sum_k (u_k u_{i \bar k} + u_{ki} u_{\bar k}),
\end{equation}
\begin{equation}
\label{cmate-C0-4}
\begin{aligned}
    \bar \partial_i \partial_i (|\nabla u|^2)
      \,& = \sum_k (u_{k\bar i} u_{\bar k i} + u_{ki} u_{\bar k\bar i}
            + u_{ki\bar i} u_{\bar k} + u_k u_{\bar k i\bar i})  \\
      \,& = \sum_k (u_{ki} u_{\bar k\bar i} + u_{i\bar i k} u_{\bar k}
           + u_{i\bar i\bar k} u_k + R_{i\bar i k\bar l} u_l u_{\bar k}) \\
      + \sum_k  \,& \Big|u_{\bar k i} -  \sum_l T^k_{il} u_{\bar l}\Big|^2
          - \sum_k \Big|\sum_l T^k_{il} u_{\bar l}\Big|^2.
\end{aligned}
\end{equation}
Therefore, by \eqref{gblq-B145} and \eqref{cmate-eq148},
\begin{equation}
\label{cmate-C0-41}
\begin{aligned}
 \sum_i S_{\alpha -1;i} (\fg^{i\bar i})^2 \bar\partial_i\partial_i(|\nabla u|^2)
   \geq \,& \sum_{i,k} S_{\alpha - 1;i}(\fg^{i\bar i})^2  |u_{ki}|^2 \\
      - C |\nabla u|^2  - C |\nabla u|^2  \sum_{i} \,& S_{\alpha - 1;i}(\fg^{i\bar i})^2.
  \end{aligned}
\end{equation}
From (\ref{cmate-C0-1}) and (\ref{cmate-C0-3}),
\begin{equation}
\label{cmate-C0-5}
\begin{aligned}
    \big|\partial_i(|\nabla u|^2)\big|^2
   = \,& \Big|\sum_k u_{ki}u_{\bar k}\Big|^2
       - 2|\nabla u|^2 \sum_k \mathfrak{Re}\big\{ u_k u_{i \bar k} \phi_{\bar i}\big\}
       - \Big|\sum_k u_k u_{i \bar k}\Big|^2 \\
\leq \,& |\nabla u|^2 \sum_k |u_{ki}|^2
       - 2 |\nabla u|^2 \sum_k \mathfrak{Re}\big\{ u_k u_{i \bar k} \phi_{\bar i}\big\}
\end{aligned}
\end{equation}
by Schwarz inequality.

Combining (\ref{cmate-C0-2}), \eqref{cmate-C0-41} and (\ref{cmate-C0-5}) we derive
\begin{equation}
    \label{cmate-C0-F}
    \begin{aligned}
\sum_i S_{\alpha -1;i} (\fg^{i\bar i})^2 (\phi_{i\bi} - C)
  + \frac{2}{|\nabla u|^2} \sum_{i,k} S_{\alpha -1;i} (\fg^{i\bar i})^2
      \mathfrak{Re}\big\{ u_k u_{i \bar k} \phi_{\bar i}\big\}
  \leq \,& C. %C |\nabla u|^2 \sum_{i} S_{\alpha - 1;i}(\fg^{i\bar i})^2.
     \end{aligned}
\end{equation}

Next,
\[ \partial_i\phi = \phi \, \partial_i \eta, \;\;
   \bar\partial_i \partial_i \phi
    = \phi \, (|\partial_i\eta|^2 + \bar\partial_i\partial_i\eta). \]
Therefore,
\begin{equation}
    \label{cmate-C0-6}
    \begin{aligned}
 2 \phi^{-1} %\sum_{i,k} S_{\alpha - 1;i}(\fg^{i\bar i})^2
              \sum_k \mathfrak{Re} \big\{u_k u_{i \bar k} \phi_{\bar i}\big\}
   \geq \,& 2 %\sum_i  S_{\alpha -1; i} \fg^{i\bar i}
                  \fg_{i\bar i} \mathfrak{Re}\big\{u_i\eta_{\bar i}\big\}
          - \frac{1}{2} |\nabla u|^2 %\sum_i \,& S_{\alpha -1;i} (\fg^{i\bar i})^2
             |\eta_i|^2  - C %\sum_i  S_{\alpha -1;i} (\fg^{i\bar i})^2
    \end{aligned}
\end{equation}
and
\begin{equation}
    \label{cmate-C0-7}
 \begin{aligned}
\sum_i S_{\alpha -1;i} (\fg^{i\bar i})^2 \eta_{i \bar i}
        \,&   + \frac{1}{2} \sum_i S_{\alpha -1;i} (\fg^{i\bar i})^2 |\eta_i|^2 \\
  \leq \,& - \frac{2}{|\nabla u|^2} \sum_i  S_{\alpha -1; i} \fg^{i\bar i}
             \mathfrak{Re}\big\{u_i\eta_{\bar i}\big\} + \frac{C}{\phi} \\
         + C \Big(\frac{1}{\phi} \,& + \frac{1}{|\nabla u|^2}\Big)
             \sum_i S_{\alpha -1;i} (\fg^{i\bar i})^2.
    \end{aligned}
 \end{equation}

For $N > 0$ sufficiently large so that Lemma~\ref{cmate-2alternative} holds,
we consider two cases: ({\bf a}) $w > N$ and ({\bf b}) $w \leq N$.
Without loss of generality we can assume that $|\nabla u| > |\nabla \underline u|$
at $p$ or otherwise we are done.
Note that
\begin{equation}
    \label{cmate-C0-8}
- \frac{2}{|\nabla u|^2} \sum_i  S_{\alpha -1; i} \fg^{i\bar i}
                   \mathfrak{Re}\big\{u_i\eta_{\bar i}\big\}
\leq 4  \sum_i  S_{\alpha -1; i} \fg^{i\bar i} = 4 \alpha S_\alpha.
\end{equation}

In case ({\bf a}) we have by Lemma~\ref{cmate-2alternative}
\begin{equation}
    \label{cmate-C1-20}
\sum_i  S_{\alpha -1;i} (\fg^{i\bar i})^2 \eta_{i \bar i}
     \geq \theta +  \theta \sum_i  S_{\alpha -1;i} (\fg^{i\bar i})^2.
\end{equation}
So if  $S_{\alpha -1;i} (\fg^{i\bar i})^2 \geq K$ for some $i$ and
$K$ sufficiently large
we derive a bound $|\nabla u| \leq C$ from \eqref{cmate-C0-7} and \eqref{cmate-C0-8}
when $A$ is sufficiently large.

Suppose that $S_{\alpha -1;i} (\fg^{i\bar i})^2 \leq K$ for all $i$
and assume $\fg_{1\bar 1} \leq \cdots \leq \fg_{n\bar n}$.
Note that
\[ %\fg^{1\bar 1} \cdots \fg_{\alpha \balpha}
   \prod_{i = 1}^{\alpha} \fg^{i\bar i} \geq \frac{S_{\alpha}}{C_n^{\alpha}}
    = \frac{1}{\psi}. \]
We have
\[ \frac{\fg^{1\bar 1}}{\psi}
   \leq (\fg^{1\bar 1})^2 \prod_{i = 2}^{\alpha} \fg^{i\bar i}
%   \fg^{2\bar 2} \cdots \fg_{\alpha \balpha}
   \leq S_{\alpha -1;1} (\fg^{1\bar 1})^2 \leq K. \]
Therefore, for all $1 \leq i \leq n$,
\[ S_{\alpha -1;i} \leq C_n^{\alpha -1} (\fg^{1\bar 1})^{\alpha -1}
   \leq C_n^{\alpha -1} (K \psi)^{\alpha -1} \leq K'.  \]
   By Schwarz inequality,
\begin{equation}
    \label{cmate-C0-8'}
   \begin{aligned}
%- \frac{2}{|\nabla u|^2}
- 2 \sum_i  S_{\alpha -1; i} \fg^{i\bar i}
                   \mathfrak{Re}\big\{u_i\eta_{\bar i}\big\}
    \leq \,& 4 \sum_i  S_{\alpha -1; i}
 + \frac{1}{4} |\nabla u|^2 \sum_i  S_{\alpha -1; i} (\fg^{i\bar i})^2 |\eta_i|^2 \\
 \leq \,& \frac{1}{4} |\nabla u|^2 \sum_i  S_{\alpha -1; i} (\fg^{i\bar i})^2 |\eta_i|^2
  + C.
  \end{aligned}
\end{equation}
From \eqref{cmate-C0-7}, \eqref{cmate-C1-20} and \eqref{cmate-C0-8'} we obtain
\[ \frac{\theta}{C} - \frac{1}{\phi} - \frac{1}{|\nabla u|^2} \leq 0. \]
This gives a bound for $|\nabla u|$ when $A$ is chosen sufficiently large.

In case ({\bf b}) we have %$w<N$,
\begin{equation}
    \label{cmate-C0-F5}
    \begin{aligned}
 \sum_i  S_{\alpha -1;i} (\fg^{i\bar i})^2 |\eta_i|^2
    %+ S_{\alpha -1;i} (\fg^{i \bar i})^2 \eta_{i\bar i}
    \,& \geq |\nabla \eta|^2 \min_i S_{\alpha -1;i} (\fg^{i\bar i})^2
     %         + \epsilon \sum_i S_{\alpha -1;i} (\fg^{i\bar i})^2 - S_{\alpha} \\
     \geq \frac{|\nabla \eta|^2}{w^{\alpha + 1}}
      \geq \frac{|\nabla \eta|^2}{N^{\alpha + 1}}.
    \end{aligned}
\end{equation}
Substituting this into (\ref{cmate-C0-7}), we derive from
\eqref{cmate-C0-8} and \eqref{cmate-eq151},
\begin{equation}
    \label{cmate-C0-F6}
    \begin{aligned}
    \frac{|\nabla \eta|^2}{2N^{\alpha + 1}}
      %+ \epsilon \sum_i  S_{\alpha -1;i} (\fg^{i\bar i})^2
   \leq \,&  5 \alpha S_\alpha
   + \frac{C}{\phi} + \Big(\frac{C}{\phi} + \frac{C }{|\nabla u|^2} - \epsilon\Big)
      \sum_i S_{\alpha -1;i} (\fg^{i\bar i})^2.
    \end{aligned}
\end{equation}
This gives a bound $|\nabla u|\leq C$.
\end{proof}

\bigskip

\section{Boundary estimates for second derivatives}
\label{gblq-B}
\setcounter{equation}{0}

\medskip

In this section we prove Theorem~\ref{gl-thm20bc2}. 
Throughout this section we assume that $\varphi$ is extended smoothly to $\bM$
and that $\ul{u} \in C^2 (\bM)$ is a subsoltuion satisfying \eqref{CH-I20b}.
As in \cite{GL10} and \cite{GL3} we follow the idea of \cite{GS93}, 
\cite{Guan98a},
\cite{Guan98b} to use $\ul{u} - u$ in construction of barrier functions.

To derive \eqref{cma-37} let us
consider a boundary point $0 \in \partial M$.
We use  coordinates around $0$ such that
$\frac{\p}{\p x_n}$ is the interior normal direction to
$\partial M$ at 0 and $g_{i\bj} (0) = \delta_{ij}$.
For convenience we set
\[ %\begin{aligned}
t_{2k-1} = x_k,\;  t_{2k} = y_k, \, 1 \leq k \leq n-1; \;
%& t_1 = x_1, \; t_2 = y_1, \; \ldots,
%\; t_{2n-3} = x_{n-1}, \; t_{2n-2} = y_{n-1}, \\
  t_{2n-1} =  y_n, \; t_{2n} = x_n . \]

Since $u - \varphi = 0$ on $\partial M$, one derives
\begin{equation}
\label{cma-70}
|u_{t_{\alpha} t_{\beta}}(0)| \leq C, \;\;\;\; \alpha, \beta < 2n
\end{equation}
where $C$ depends on $|u|_{C^1 (\bM)}$, $|\ul{u}|_{C^1 (\bM)}$, and
geometric quantities of $\partial M$.

To estimate $u_{t_{\alpha} x_n} (0)$ for $\alpha \leq 2n$,
we shall employ a barrier function of the form
\begin{equation}
\label{cma-E85}
v = (u - \ul{u}) + t \sigma - T \sigma^2
\;\; \mbox{in $\Omega_{\delta} = M \cap B_{\delta}$}
\end{equation}
where $t, T$ are positive constants to be determined,
$B_{\delta}$ is the
(geodesic) ball of radius $\delta$ centered at $p$,
and $\sigma$ is the distance function to $\partial M$.
Note that $\sigma$ is smooth in
$M_{\delta_0} := \{z \in M: \sigma (z) < \delta_0\}$
for some $\delta_0 > 0$.
%which we call the $\delta_0$-neighborhood of $\partial M$.

\begin{lemma}
\label{cma-lemma-20}
There exists $c_0 > 0$ such that for $T$ sufficiently large and
$t, \delta$ sufficiently small,
$v \geq 0$ and
\begin{equation}
\label{eq-100}
%\begin{aligned}
 \sum_{i,j} F^{i\bj} v_{i\bj}
    \leq - c_0 \Big(1 + \sum_{i,j} F^{i\bj} g_{i\bj}\Big)
   \;\; \mbox{in} \;\; \Omega_{\delta}.
%  v & \, \geq 0 \;\; \mbox{on} \;\; \partial \Omega_{\delta}
%  \end{aligned}
\end{equation}
\end{lemma}

\begin{proof}
The proof is very similar to that of Lemma~5.1 in \cite{GL3}; for
completeness we include it here.
First of all, since $\sigma$ is smooth and $\sigma = 0$ on $\partial M$,
for fixed $t$ and $T$ we may require $\delta$ to be so small that
$v \geq 0$ in $\Omega_{\delta}$. Next, note that
\[ \sum_{i,j} F^{i\bj} \sigma_{i\bj} \leq C_1 \sum_{i,j} F^{i\bj}  g_{i\bj} \]
for some constant $C_1 > 0$ under control. Therefore,
\begin{equation}
\label{eq-110}
  \sum_{i,j} F^{i\bj} v_{i\bj} \leq \sum_{i,j} F^{i\bj} (u_{i\bj} -\ul{u}_{i\bj})
    + C_1 (t + T \sigma) \sum_{i,j} F^{i\bj} g_{i\bj}
    - 2 T \sum_{i,j} F^{i\bj} \sigma_i \sigma_{\bj}.
\end{equation}

Fix $N > 0$ sufficiently large so that Lemma~\ref{cmate-2alternative}
holds. At a fixed point in $\Omega_{\delta}$, we consider two cases:
(a) $w \leq N$ and (b) $w > N$.

In case (a) let $\lambda_1 \leq \cdots \leq \lambda_n$ be
the eigenvalues of $\{\fg_{i\bj}\}$. %(with respect to $\{g_{i\bj}\}$).
We see from equation~\eqref{cmate-M10'}
that there is a uniform lower bound $\lambda_1 \geq c_1 > 0$.
Consequently, $c_2 I \leq \{F^{i\bj}\} \leq \frac{1}{c_2} I$ for some
constant $c_2 > 0$ depending on $N$ and $c_1$, and hence
\begin{equation}
\label{cma-E95}
\sum_{i,j} F^{i\bj} \sigma_i \sigma_{\bj} \geq c_2 |\nabla \sigma|^2 = \frac{c_2}{4}.
\end{equation}
Since $F$ is homogeneous of degree one, by \eqref{eq-110}, \eqref{cma-E95}
and \eqref{cmate-eq151},
\begin{equation}
\label{cma-E90}
\sum_{i,j} F^{i\bj} v_{i\bj} \leq F (\fg_{i\bj})
      + (C_1 (t + T \sigma) - \epsilon) \sum_{i,j} F^{i\bj} g_{i\bj} - \frac{c_2 T}{2}
  \leq - \frac{\epsilon}{2} \Big(1 + \sum_{i,j} F^{i\bj} g_{i\bj}\Big)
\end{equation}
if we fix $T$ sufficiently large and require $t$ and $\delta$ small to satisfy
$C_1 (t + T \delta) \leq \epsilon/2$.

Suppose now that $w > N$.
By Lemma~\ref{cmate-2alternative} and \eqref{eq-110}, we may further
require $t$ and $\delta$ to satisfy $C_1 (t + T \delta) \leq \theta/2$
so that \eqref{eq-100} holds.
\end{proof}

Using Lemma~\ref{cma-lemma-20} we may derive as in \cite{GL10} (but see \cite{GL3} for
 some corrections) the estimates $|u_{t_{\alpha} x_n} (0)| \leq C$ (and therefore $|u_{x_n t_{\alpha}} (0)| \leq C$) for $\alpha < 2n$; we shall omit the proof here.
 It remains to prove
\begin{equation}
    \fg_{n\bar n} (0) \leq C .
\end{equation}
The proof below uses an idea of Trudinger~\cite{Trudinger95}.

Let $T_C \partial M$ be the complex tangent bundle and
\[ T^{1,0}\partial M = T^{1,0} M \bigcap T_C \partial M
                     = \{\xi\in T^{1,0}M : d \sigma(\xi) = 0\}. \]
Let $\hat \chi_u$ and $\hat \omega$ denote the restrictions to $T_C \partial M$ of $\chi_u$
and $\omega$ respectively. As in \cite{GL3} we only have to show that
$$
m_0 := \min_{\partial M} \frac{n \hat x^{n-1}_u}{\psi (n - \alpha) \hat \chi^{n - \alpha - 1}_u \wedge \hat \omega^\alpha} > 1 .
$$

Suppose that $m_0$ is reached at a point $0 \in \partial M$. 
Let $\tau_1, \cdots, \tau_{n-1}$ be a local frame of vector fields in $T^{1,0}_C \partial M$ around 0 such that $g(\tau_\beta, \bar\tau_\gamma) = \delta_{\beta\gamma}$ for $1 \leq \beta, \gamma \leq n - 1$ and
$\tau_\beta = \frac{\partial} {\partial z_{\beta}}$ at $0$.
We extend  $\tau_1, \ldots, \tau_{n-1}$
 by their parallel transports
along geodesics normal to $\partial M$ so that they are smoothly defined in
a neighborhood of $0$.  Denote
$\tilde{u}_{\beta \bar{\gamma}} = u_{\tau_{\beta} \bar{\tau}_{\gamma}}$
and $\tilde{\fg}_{\beta \bar{\gamma}} = \tilde{u}_{\beta \bar{\gamma}}
+ \chi (\tau_{\beta}, \bar{\tau}_{\gamma})$, $1 \leq \beta, \gamma \leq n - 1$,
etc.
On $\partial M$ we have
\begin{equation}
  \frac{n\hat \chi^{n-1}_u}{\psi(n-\alpha)\hat\chi^{n-\alpha-1}_u \wedge \hat\omega^\alpha}
  = \frac{C^\alpha_n}{\psi} \frac{S_{n-1} (\tilde{\fg}_{\beta \bar{\gamma}})}
      {S_{n-\alpha-1}(\tilde{\fg}_{\beta \bar{\gamma}})}.
\end{equation}

Define, for a positive definite $(n-1)\times(n-1)$ Hermitian matrix $\{r_{\beta\bar\gamma}\}$,
$$
G[r_{\beta \bar\gamma}] := \left(\frac{S_{n-1}(\lambda (r_{\beta \bar\gamma}))}
      {S_{n-\alpha-1} (\lambda (r_{\beta \bar\gamma}))}\right)^{\frac{1}{\alpha}},
$$
where $\lambda (r_{\beta \bar\gamma})$ denotes the ordinary eigenvalues of 
$\{r_{\beta \bar\gamma}\}$ (with respect to the identity matrix $I$),
 and let
$$
G^{\beta\bar\gamma}_0
    = \frac{\partial G}{\partial r_{\beta \bar\gamma}} [\fg_{\beta \bar\gamma}(0)] .
$$
Note that $G$ is concave and homogeneous of degree one. Therefore,
\begin{equation}
\label{cma-701}
 \sum_{\beta,\gamma < n} G^{\beta \bar\gamma}_0 r_{\beta \bar\gamma}
   \geq G [r_{\beta \bar\gamma}]
\end{equation}
for any $\{r_{\beta \bar\gamma}\}$. In particular, since
$u_{\beta \bar\gamma}(0) = \underline u_{\beta \bar\gamma} (0)
    + (u- \underline u)_{x_n} (0) \sigma_{\beta\bar\gamma} (0)$, we have
\begin{equation}
\label{cma-702}
\begin{aligned}
 G[\fg_{\beta \bar\gamma} (0)]
   = \,& \sum_{\beta,\gamma < n} G^{\beta \bar\gamma}_0 \fg_{\beta \bar\gamma}(0) \\
   = \,& \sum_{\beta,\gamma < n} G^{\beta \bar\gamma}_0
      (\chi_{\beta \bar\gamma}(0) + \underline u_{\beta \bar\gamma}(0))
      + (u - \underline u)_{x_n} (0)
      \sum_{\beta,\gamma < n} G^{\beta \bar\gamma}_0 \sigma_{\beta \bar\gamma}(0).
      \end{aligned}
\end{equation}

We shall need the following elementary lemma. 

\begin{lemma}
\label{cma-lemma-30}
Let %$A$ be a positive definite Hermitian matrix
\[ A = \left[ \begin{aligned} B \;\;\; & \; C \\ 
                         \bar{C}' \;\; & a_{n\bn} \end{aligned} \right] \]
be a positive definite Hermitian matrix. Then
%where $\bar{C}$ is the conjugate transpose of $C$. Then
%There exists $c_0 > 0$ depending on the lower and upper bounds of the eigenvalues of $A$
%such that
\begin{equation}
\label{eq-101}
G^{\alpha} (B) \geq (1+ c_0) \frac{S_{n}(\lambda (A))}
      {S_{n-\alpha} (\lambda (A))}
\end{equation}
where $c_0 > 0$ depends on the lower and upper bounds of the eigenvalues of $A$.
\end{lemma}

\begin{proof}
It is straightforward to verify that
\[ \left[ \begin{aligned} I \;\;\;\; & 0 \\ 
                \bar{C}' B^{-1} \;\; & 1 \end{aligned} \right] \;
     \left[ \begin{aligned} B \;\;\; & \; C \\ 
                       \bar{C}' \;\; & a_{n\bn} \end{aligned} \right] \;
       \left[ \begin{aligned} I \;\; & B^{-1} C \\ 
                              0 \;\; & \;\;\; 1 \end{aligned} \right]
     = \left[ \begin{aligned} B \;\; & \;\;\;\;\;\; 0 \\ 
               \;\; 0 \;\; & a_{n\bn} - \bar{C}' B^{-1} C \end{aligned} \right]. \]
So 
\[ \det A %= (a_{n\bn} - \bar{C}' B^{-1} C) S_{n-1} (B) 
          = (a_{n\bn} - \bar{C}' B^{-1} C) \det B. \]
We now claim 
\[ S_{n-\alpha} (\lambda (A)) 
   \geq (a_{n\bn} - \bar{C}' B^{-1} C) S_{n-\alpha-1} (\lambda (B))
         + S_{n-\alpha} (\lambda (B)). \]
To see this we can assume $B$ is diagonal and consider a submatrix of $A$ of the form
\[ A_J = \left[ \begin{aligned} B_J \;\;\; & \; C_J \\
                         \bar{C_J}' \;\; & a_{n\bn} \end{aligned} \right]. \]
We have 
\[ \bar{C}' B^{-1} C \geq \bar{C_J}' B_J^{-1} C_J \geq 0 \]
since $B$ is positive definite and $\bar{C}'$ is the conjugate transpose of $C$.
Therefore,
\[ \det A_J = (a_{n\bn} - \bar{C_J}' B_J^{-1} C_J) \det B_J
            \geq (a_{n\bn} - \bar{C}' B^{-1} C) \det B_J. \]
The claim and \eqref{eq-101} now follow easily. 
\end{proof}

We continue the proof of \eqref{cma-37}. 
Suppose that for some small $\theta_0 > 0$ to be determined later,
$$
- \sum_{\beta,\gamma < n} (u - \underline u)_{x_n} (0)
        G^{\beta \bar\gamma}_0 \sigma_{\beta \bar\gamma}(0)
   \leq \theta_0 \sum_{\beta,\gamma < n} G^{\beta\bar\gamma}_0 (\chi_{\beta\bar\gamma}(0) + \underline u_{\beta\bar\gamma}(0)).
$$
Then,
\begin{equation}
    \begin{aligned}
    G[\fg_{\beta\bar\gamma}(0)]
    \geq \,& (1 - \theta_0) \sum_{\beta,\gamma < n} G^{\beta\bar\gamma}_0
              (\chi_{\beta\bar\gamma}(0) + \underline u_{\beta\bar\gamma}(0)) \\
    \geq \,& (1 - \theta_0) \,
             G[\chi_{\beta \bar\gamma}(0) + \underline u_{\beta \bar\gamma}(0)] \\
               %=& (1 - \theta_0) \left( \frac{S_{n-1}(\chi_{\beta\bar\gamma(0)}
               %+ \underline u_{\beta\bar\gamma}(0))}{S_{n- \alpha -1}
               %(\chi_{\beta\bar\gamma}(0) + \underline u_{\beta\bar\gamma}(0))}
                     %\right)^{\frac{1}{\alpha}} \\
     \geq \,& (1 - \theta_0) (1 + c_0) F (\chi_{\ul u}) \\
                %\left(\frac{S_n (\chi_{\underline u})}
                %{S_{ n-\alpha}(\chi_{\underline u})}\right)^{\frac{1}{\alpha}} \\
      \geq \,& (1 - \theta_0) (1 + c_0) 
      \left(\frac{\psi(0)}{C^\alpha_n}\right)^{\frac{1}{\alpha}}.
    \end{aligned}
\end{equation}
The second and fourth inequalities follow  from \eqref{cma-701}
and \eqref{CH-I20b}, respectively,
while the third from Lemma~\ref{cma-lemma-30}. 
Choosing $\theta_0$ small enough, we obtain
$$
m_0 = \frac{C_n^\alpha}{\psi(0)}  G[\fg_{\beta\bar\gamma}(0)]^{\alpha}
    \geq 1 + \frac{\theta_0}{2}.
$$

Suppose now that
$$
- (u - \underline u)_{x_n} (0)
     \sum_{\beta,\gamma < n} G^{\beta \bar\gamma}_0 \sigma_{\beta \bar\gamma} (0)
   > \theta_0 \sum_{\beta,\gamma < n}  G^{\beta \bar\gamma}_0
      (\chi_{\beta \bar\gamma}(0) + \underline u_{\beta \bar\gamma}(0)).
$$
On $\partial M$,
$\tilde u_{\beta\bar\gamma} = \tilde \varphi_{\beta\bar\gamma} 
    + (u - \varphi)_\nu \tilde \sigma_{\beta\bar\gamma}$
where
\begin{equation*}
    \nu = \sum^{2n}_{k = 1} \nu^k \frac{\p}{\p t_k}
\end{equation*}
is the interior unit normal vector field to $\partial M$.
We have $|\nu^k| \leq C\rho$ for $k < 2n$ and $|(u - \varphi)_{t_k}| \leq C \rho$
since $\nu^k (0) = 0$ for $k < 2n$ and $u = \varphi$ on $\partial M$. Define
\begin{equation}
    \begin{aligned}
 \varPhi = \,& \sum_{\beta,\gamma < n} G^{\beta \bar\gamma}_0
            (\tilde \chi_{\beta \bar\gamma} + \tilde\varphi_{\beta \bar\gamma})
            + (u - \varphi)_{x_n} \nu^{2n} \sum_{\beta,\gamma < n}
            G^{\beta\bar\gamma}_0 \tilde \sigma_{\beta \bar\gamma}
            - \Big(\frac{m_0 \psi}{C^\alpha_n}\Big)^{\frac{1}{\alpha}} \\
      := \,& - (u - \varphi)_{x_n} \eta + Q
    \end{aligned}
\end{equation}
where $\eta$
 %:= \nu^{2n}\sum_{\beta,\gamma} G^{\beta\bar\gamma}_0 \sigma_{\beta\bar\gamma}$
 and Q are smooth. Note that $\varPhi(0) = 0$ and
\begin{equation}
   \begin{aligned}
 \eta (0)
  = - \,& \nu^{2n} (0) \sum_{\beta,\gamma < n} G^{\beta \bar\gamma}_0
                          \sigma_{\beta \bar\gamma} (0) \\
  > \,& \frac{\theta_0}{(u - \ul u)_{x_n} (0)}
         \sum_{\beta,\gamma < n} G^{\beta \bar\gamma}_0
          (\chi_{\beta \bar\gamma}(0) + \ul u_{\beta \bar\gamma}(0)) \\
\geq \,& \frac{\theta (1 + \epsilon) \psi(0)}{C^\alpha_n (u - \underline u)_{x_n} (0) }
          \geq c_2 >  0.
       \end{aligned}
\end{equation}
On $\partial M$,
\begin{equation}
    \begin{aligned}
\varPhi 
 = \,& \sum_{\beta,\gamma < n}G^{\beta \bar\gamma}_0 \tilde \fg_{\beta \bar\gamma}
       - \sum_{k < 2n}(u - \varphi)_{t_k} \nu^{k}
     \sum_{\beta,\gamma < n} G^{\beta \bar\gamma}_0 \tilde \sigma_{\beta \bar\gamma}
            - \Big(\frac{m_0 \psi}{C^\alpha_n}\Big)^{\frac{1}{\alpha}} \\
  \geq &  \sum_{k < 2n}(u - \varphi)_{t_k} \nu^{k}
     \sum_{\beta,\gamma < n} G^{\beta \bar\gamma}_0 \tilde \sigma_{\beta \bar\gamma}
  \geq - C \rho^2
    \end{aligned}
\end{equation}
since by \eqref{cma-701}
\[ \sum_{\beta,\gamma < n}G^{\beta \bar\gamma}_0 \tilde \fg_{\beta \bar\gamma}
    \geq G [\tilde \fg_{\beta \bar\gamma}]
    \geq \Big(\frac{m_0 \psi}{C^\alpha_n}\Big)^{\frac{1}{\alpha}}. \]

We calculate
\begin{equation}
    \begin{aligned}
 & \sum_{i,j} F^{i\bar j} \varPhi_{i\bar j}
   \leq - \eta \sum_{i,j} F^{i\bar j}  (u_{x_n})_{i\bar j}
            + C \sum_{i,j} F^{i\bar j} g_{i\bar j}  \\
          + \sum_{i,j} F^{i\bar j} \,& (u - \varphi)_{x_n z_i} (u - \varphi)_{x_n \bz_j}.
\end{aligned}
\end{equation}
As in \cite{CKNS} (see also \cite{GL3}),
\[ \sum_{i,j} F^{i\bar j} (u - \varphi)_{x_n z_i} (u - \varphi)_{x_n \bz_j}
\leq \sum_{i,j} F^{i\bar j} (u - \varphi)_{y_n z_i} (u - \varphi)_{y_n \bz_j}
+  C \sum_{i,j} F^{i\bar j} g_{i\bar j} + C. \]
On the other hand, differentiating equation \eqref{cmate-M10} with respect to $x_n$, 
we see that
\begin{equation}
\label{cmate-bd-E1}
    \begin{aligned}
  - \sum_{i,j} F^{i\bar j} (u_{x_n})_{i\bar j}
  \leq \,& 2 \Big|\sum_{i,j,l} F^{i\bar j} \fg_{i\bar l} \overline{\Gamma^l_{nj}}\Big|
          + C \sum_{i,j} F^{i\bar j} g_{i\bar j} + C.
      %+ 2 \Big|\sum_{i,j,l}\Gamma^l_{ni} F^{i\bar j} \fg_{l\bar j}\Big|.
    \end{aligned}
\end{equation}
At a fixed point choose a unitary $A = \{a_{ij}\}_{n \times n}$ which diagonalizes
$\{\fg_{i\bj}\}$. We have
\begin{equation}
\label{cmate-bd-E2}
    \begin{aligned}
    \sum_{i,j,l} F^{i\bar j} \fg_{i\bar l} \overline{\Gamma^l_{nj}}
    = \,& \sum_{i,j,l,s,t,p,q} a^{is} f_s \delta_{st} \bar a^{jt} a_{ip} \lambda_p
        \delta_{pq} \bar a_{lq} \overline{\Gamma^l_{nj}} \\
    = & \sum_{q} f_q \lambda_q \sum_{j,l} \bar a^{jq} \bar a_{lq} \ol{\Gamma^l_{nj}}
      \leq C \psi .
\end{aligned}
\end{equation}\
Therefore, %So substituting \eqref{cmate-bd-E2} into \eqref{cmate-bd-E1}
\begin{equation}
 - \sum_{i,j} F^{i\bar j} (u_{x_n})_{i\bar j}
  \leq C \sum_{i,j} F^{i\bar j} g_{i\bar j} + C.
      %+ 2 \Big|\sum_{i,j,l}\Gamma^l_{ni} F^{i\bar j} \fg_{l\bar j}\Big|.
\end{equation}

Applying Lemma~\ref{cma-lemma-20} we derive
\[ \sum_{i,j} F^{i\bj} (Av + B\rho^2 + \Phi - |(u - \varphi)_{y_n}|^2)_{i\bj}
    \leq 0 \;\; \mbox{in $M\cap B_\delta(0)$} \]
and $Av + B\rho^2 + \Phi - |(u - \varphi)_{y_n}|^2 \geq 0$ on
$\partial (M \cap B_\delta(0))$ when $A \gg B \gg 1$.
By the maximum principle,
$Av + B\rho^2 + \Phi - |(u - \varphi)_{y_n}|^2 \geq 0$ in $M\cap B_\delta(0)$,
and therefore $\Phi_{x_n}(0) \geq - C$. This gives
\[ u_{n\bar n} (0) \leq C. \]

We now have positive lower and upper bounds for all eigenvalues of $\{\fg_{i\bj} (0)\}$.
By Lemma~\ref{cma-lemma-30}, 
\[ G[\fg_{\beta \bar \gamma} (0))] \geq (1 + c_0) F (\fg_{i\bj} (0)) \]
for some $c_0 > 0$. It follows that
\[    m_0 = \frac{C_n^\alpha}{\psi(0)} G[\fg_{\beta \bar \gamma} (0))]
     \geq  1 + c_0. \]
The proof of \eqref{cma-37} is therefore complete.

\bigskip

\section{The second order estimates}
\label{cmate-C2}
\setcounter{equation}{0}

\medskip

\begin{proposition}
 Suppose $\chi \in \mathscr{C}_{\alpha} (\omega)$ and 
let $u \in C^4(M) \cap C^2(\bM)$ be a solution of equation \eqref{CH-I10}.
Then there is a uniform constant $C>0$ such that
\begin{equation}
\label{cmate-C2-1}
\sup_{\ol M} \Delta u \leq C(1 + \sup_{\partial M} \Delta u).
\end{equation}
%where $C$ depends on the given geometric data.
\end{proposition}

\begin{proof}
Let $\phi$ be a function to be determined later and assume that
$w e^{\phi}$  reaches its maximum at some point $p \in M $
where $w = \Delta u + tr\chi$. 
Choose local coordinates around $p$ such that
$g_{i\bar j}(p) = \delta_{ij}$ and $\fg_{ij}$ is diagonal. At $p$ we have
\begin{equation}
\label{cmate-C2-F1}
\frac{\partial_l w}{w} + \partial_l\phi = 0,  \;\;
\frac{\bar\partial_l w}{w} + \bar\partial_l\phi = 0
\end{equation}
and
\begin{equation}
    \label{cmate-C2-F}
\frac{\bar\partial_l\partial_l w}{w} - \frac{\bar\partial_l w\partial_l w}{w^2} 
    + \bar\partial_l\partial_l\phi \leq 0.
\end{equation}

By \eqref{cmate-C2-F1} and Schwarz inequality,
\begin{equation}
    \label{cmate-C2-2}
  \begin{aligned}
    \big|\partial_l w \big|^2
    = \,& \Big|\sum_i \fg_{i\bar i l} \Big|^2
           = \Big|\sum_i (\fg_{l\bar ii} - T^i_{li} \fg_{i\bar i}) + \lambda_l\Big|^2 \\
 \leq \,& w \sum_i \fg^{i\bar i} \big| \fg_{l\bar ii} - T^i_{li} \fg_{i\bar i}\big|^2
          - 2 w \mathfrak{Re} \big\{\phi_l \bar{\lambda_l} \big\} 
          - \big|\lambda_l\big|^2
 \end{aligned}
\end{equation}
where
\begin{equation*}
    \lambda_l = \sum_i \Big(\chi_{i\bar il} - \chi_{l\bar ii} 
                + \sum_j T^j_{li}\chi_{j\bar i}\Big).
\end{equation*}

Next, by \eqref{gblq-R155} and \eqref{cmate-eq149},
\begin{equation}
\label{cmate-C2-E1}
\begin{aligned}
    \sum_l S_{\alpha-1;l} (\fg^{l\bar l})^2 \bar\partial_l\partial_l w
      = \, & \sum_{i,l} S_{\alpha-1;l} (\fg^{l\bar l})^2 \fg_{i\bar i l\bar l} \\
    \geq \,& \sum_{i,l} S_{\alpha-1 ;l} (\fg^{l\bar l})^2 \fg_{l\bar li\bar i}
              - 2 \sum_{i,j,l} S_{\alpha-1 ;l} (\fg^{l\bar l})^2
               \mathfrak{Re} \big\{\overline{T^j_{li}}\fg_{l\bar ji}\big\} \\
           & + \sum_{i,j,l} S_{\alpha-1 ;l} (\fg^{l\bar l})^2
               T^j_{li} \overline{T^j_{li}} \fg_{j\bar j}
               - C w \sum_l S_{\alpha-1;l} (\fg^{l\bar l})^2 \\
                %  R_{i\bar il\bar l}\fg_{l\bar l} + R_{l\bar li\bar i}\fg_{i\bar i}
                %  + G_{l\bar li\bar i}\Big) .
    \geq \,&  \sum_{i,j,l} S_{\a -1;i}(\fg^{i\bi})^2\fg^{j\bj}
                  \big|\fg_{i\bj l} - T^j_{il}\fg_{j\bj}\big|^2 \\
               &  - C w \sum_l S_{\alpha-1;l} (\fg^{l\bar l})^2 - C.
\end{aligned}
\end{equation}
It follows from (\ref{cmate-C2-F}),  (\ref{cmate-C2-2}) and (\ref{cmate-C2-E1}) that
\begin{equation}
    \label{cmate-C2-F2}
    \begin{aligned}
    0 \geq \,& w \sum_i S_{\alpha -1;i}(\fg^{i\bar i})^2 \phi_{i\bar i}
               + 2 \sum_i S_{\alpha -1;i} (\fg^{i\bar i})^2 \mathfrak{Re}
                 \big\{\phi_i \bar {\lambda_i}\big\} \\
             & - C  w \sum_i S_{\alpha -1;i}(\fg^{i\bar i})^2 - C.
       \end{aligned}
\end{equation}

Let $\phi = e^{A\eta}$ with $\eta = \ul u - u + \sup_M (u - \ul u)$, where
$\ul u \in C^2 (\bM)$ satisfies $\chi_{\ul u}> 0$ and \eqref{cmate-eq14},
 and $A$ is a positive constant to be determined.
So
\[      \phi_i = A \phi \eta_i, \;\;
   \phi_{i\bi} = A \phi \eta_{i\bi} + A^2 \phi |\eta_i|^2. \]
Applying Schwarz inequality again,
\begin{equation}
\label{cmate-C2-3}
\begin{aligned}
	2 \sum_i S_{\alpha -1;i} (\fg^{i\bar i})^2 \fRe \big\{\phi_i \bar{\lambda_ i}\big\}
	 = \,& 2 A \phi \sum_i S_{\alpha -1;i} (\fg^{i \bar i})^2
            \mathfrak{Re} \big\{\eta_i \bar{\lambda_ i}\big\}  \\
  \geq - w A^2 \phi \sum_i \,& S_{\alpha -1;i}  (\fg^{i\bar i})^2 |\eta_i|^2
           - \frac{C \phi}{w}  \sum_i S_{\alpha -1;i} (\fg^{i\bar i})^2.
\end{aligned}
\end{equation}
Finally, by (\ref{cmate-C2-F2}) and (\ref{cmate-C2-3}),
\begin{equation}
\label{cmate-C2-F3}
\begin{aligned}
	w A \sum_i S_{\alpha -1;i}(\fg^{i\bar i})^2 \eta_{i\bar i}
   \leq \,& \frac{C}{\phi} + C \Big(\frac{1}{w} + \frac{w}{\phi}\Big)
       \sum_i S_{\alpha -1;i}(\fg^{i\bar i})^2.
\end{aligned}
\end{equation}
From Lemma~\ref{cmate-2alternative}, this gives a bound $w \leq C$ at $p$ for $A$
sufficiently large.
\end{proof}

\end{document}